\documentclass[12pt]{amsart}
\usepackage{amsmath,amssymb,amsthm,mathrsfs,hyperref,tikz}

\newcommand{\Res}{\mathrm{Res}} 
\newcommand{\Ind}{\mathrm{Ind}} 
\newcommand{\cind}{\textrm{c-}\mathrm{Ind}} 
\newcommand{\Hom}{\mathrm{Hom}}

\newcommand{\R}{{\scriptstyle{\mathcal{O}}}}
\newcommand{\PP}{\mathfrak{p}}
\newcommand{\Zb}{Z_{\mathrm{b}}}
\newcommand{\ratk}{F}
\newcommand{\extk}{E}
\newcommand{\unk}{F^{un}}
\newcommand{\p}{\varpi}

\newcommand{\SL}{\mathrm{SL}}
\newcommand{\Sp}{\mathrm{Sp}}
\newcommand{\GL}{\mathrm{GL}}
\newcommand{\Gal}{\mathrm{Gal}}
\newcommand{\Z}{\mathbb{Z}} 

\newcommand{\GG}{\mathbf{G}} 
\newcommand{\LieG}{\mathfrak{g}}
\newcommand{\mz}{\mathfrak{z}}
\newcommand{\Lie}{\mathrm{Lie}}
\newcommand{\TT}{\mathbf{T}} 
\newcommand{\Stor}{\mathbf{S}} 
\newcommand{\Ztor}{\mathbf{Z}} 
\newcommand{\LieT}{\mathfrak{t}}

\newcommand{\Bset}{\Omega}
\newcommand{\apart}{{\mathcal{A}}}
\newcommand{\Aset}{\Omega_{\apart}}
\newcommand{\buil}{\mathcal{B}}
\newcommand{\C}{\mathcal{C}} 
\newcommand{\F}{\mathcal{F}} 

\newcommand{\KT}{\mathsf{K}}
\newcommand{\J}{\mathsf{J}}

\newcommand{\bmat}[1]{\begin{bmatrix} #1 \end{bmatrix}}
\newcommand{\ep}{\varepsilon}

\theoremstyle{plain}
\newtheorem{theorem}{Theorem}[section]
\newtheorem*{theorem*}{Theorem}

\newtheorem{lemma}[theorem]{Lemma}
\newtheorem{proposition}[theorem]{Proposition}
\newtheorem{cor}[theorem]{Corollary}

\theoremstyle{definition}
\newtheorem{definition}[theorem]{Definition}
\newtheorem{remark}[theorem]{Remark}

\theoremstyle{remark}
\newtheorem{egr}[theorem]{Example}
\newenvironment{example}{\begin{egr}}{\qedsymbol \end{egr}}

\numberwithin{equation}{section}
\numberwithin{figure}{section}

\begin{document}

\title{On the unicity of types for toral supercuspidal representations}
\author{Peter Latham}
\address{Department of Mathematics, King's College London, Strand, London, UK}
\email{peter.latham@kcl.ac.uk}

\author{Monica Nevins}
\address{Department of Mathematics and Statistics, University of Ottawa, Ottawa, Canada}
\email{mnevins@uottawa.ca}
\thanks{The second author's research is supported by NSERC Canada Discovery Grant 06294.}

\date{\today}

\begin{abstract}
For tame arbitrary-length toral, also called positive regular, supercuspidal representations of a simply connected and semisimple $p$-adic group $G$, constructed as per Adler-Yu, we determine which components of their restriction to a maximal compact subgroup are types.  We give conditions under which there is a unique such component, and then present a class of examples for which there is not, disproving the strong version of the conjecture of unicity of types on maximal compact open subgroups.  We restate the unicity conjecture, and prove it holds for the groups and representations under consideration under a mild condition on depth.
\end{abstract}

\maketitle

\section{Introduction}

Let $G$ be a connected reductive $p$-adic group.  
In \cite{Bernstein1984}, J.~Bernstein gives a decomposition of the category of smooth 
representations of $G$ into indecomposable full subcategories, called \emph{blocks}, that are indexed
by the inertial support of the irreducible representations they contain. Given an irreducible supercuspidal representation $\pi$ of $G$, a \emph{type} for $\pi$ is a pair $(K, \lambda)$ consisting of an
irreducible representation $\lambda$ of a compact open subgroup $K$ of $G$ such that $\Ind_K^G \lambda$ is a projective generator of the block containing $\pi$.  In this case, $\lambda$ occurs as a subrepresentation of $\pi|_K$ and we say $\pi$ \emph{contains} the type $(K,\lambda)$.  Types are now known to exist for many classes
of supercuspidal representations.  In particular, the work of J.~Adler \cite{Adler1998}, generalized by J.K.~Yu \cite{Yu2001}, shows that every essentially tame
supercuspidal representation contains a type; the work of J.-L.~Kim \cite{Kim2007} assures us that if the residual characteristic $p$ is sufficiently large, then all supercuspidal representations of $G$ are of this form.

Given a supercuspidal representation $\pi$ of $G$ containing a type $(K,\lambda)$, it is simple to
produce additional types: any $G$-conjugate of $(K,\lambda)$ is a type, as
is any pair $(K', \tau)$ where $K'$ is a compact open subgroup of $G$ containing $K$ and $\tau$ is
an irreducible representation of $K'$ that contains $\lambda$ upon restriction to $K$. A natural
question to ask is whether $\pi$ can contain any additional types, specifically on \emph{maximal} compact open subgroups, that are not related to $(K,\lambda)$ in this way.  It is expected that this should never happen, and the conjecture that
this is the case is known as the \emph{unicity of types}.  The name is due to V.~Pa\u{s}k\={u}nas \cite{Paskunas2005}, whose thesis extended G.~Henniart's appendix ``Sur l'unicit\'e des types pour $\GL_2$'' of the article \cite{BreuilMezard2002}.  The goal of this paper is to establish
the unicity of types for a class of essentially tame supercuspidal representations which we call toral, defined below.

The unicity conjecture is a theorem for $G=\GL_n(\ratk)$ \cite{Paskunas2005}, $G=\SL_2(\ratk)$ \cite{Latham2016}, for essentially tame representations of $G=\SL_n(\ratk)$ \cite{Latham2018}, and for depth-zero supercuspidal representations of any connected reductive $p$-adic group $G$ \cite{Latham2017}.  
In each of these cases, it was seen that a stronger property holds, namely,  that if $K\subset G$ is any maximal compact open subgroup,
then there exists, up to $G$-conjugacy, at most one type defined on $K$ for each supercuspidal
representation $\pi$ of $G$.  We will refer to this as the \emph{strong unicity property}.  While strong unicity implies unicity, in Section~\ref{S:nonunicity} we provide counterexamples to prove that they are in fact inequivalent in general.

From now on, let us specialize to the case that $G$ is a simply connected semisimple $p$-adic group.  Essentially tame supercuspidal representations are constructed from sequences of twisted Levi subgroups that split over a tamely ramified Galois extension, together with characters of these subgroups and a representation of the smallest twisted Levi subgroup.  We restrict our attention here to  those sequences for which the smallest twisted Levi subgroup is an anisotropic (also called elliptic) maximal torus of $G$.    
For the purposes of this paper we call these \emph{toral} supercuspidal representations, though we caution the reader that some authors reserve ``toral'' to mean the more restrictive case that the twisted Levi sequence has length $d=1$.  Relating to work of T.~Kaletha \cite{Kaletha},  F.~Murnaghan calls our representations ``positive regular'' (as justified in \cite{Murnaghan2017}).

The strategy for proving the unicity conjecture for supercuspidal representations of $G$ is as follows.  Let $\pi$ be an essentially tame supercuspidal representation and $(\KT,\kappa)$ a type as arising from the above construction.  Since $G$ is semisimple and simply connected, we have both that  $\cind_{\KT}^G\kappa$ is irreducible (hence equivalent to $\pi$) and that 
every maximal compact open subgroup of $G$ is the stabilizer $G_y$ of a vertex $y$ in the Bruhat-Tits building of $G$.  Thus for any such $G_y$ containing $\KT$, it follows directly that we have an induced type $(G_y,\Ind_{\KT}^{G_y}\kappa)$.

The conjecture of unicity of types therefore amounts to the statement that, up to $G$-conjugacy, all types for $\pi$ on a maximal compact open subgroup arise in this way.  
Strong unicity is the statement that furthermore any two types for $\pi$ on $G_y$ are conjugate by an element of $N_G(G_y)$; this is equivalent to the statement that $(\KT,\kappa)$ is not contained two distinct but conjugate maximal compact open subgroups.

 The restriction of $\pi$ to a maximal compact open subgroup $G_y$ 
  decomposes as an infinite direct sum of irreducible representations of $G_y$.  Describing these branching rules is a difficult open problem of interest in its own right.  Here, it suffices to note that with $G$ as above, the Bernstein block corresponding to $\pi$ is generated by $\pi$.  Therefore the types of $\pi$ supported on $G_y$ are exactly those irreducible components of $\pi|_{G_y}$ that do not occur in $\pi'|_{G_y}$ for any other (inequivalent) irreducible representation $\pi'$ of $G$.  Proving this, in turn, is made possible by the major work of J.~Hakim and F.~Murnaghan \cite{HakimMurnaghan2008} which establishes the equivalences among essentially tame supercuspidal representations entirely in terms of the Adler-Yu data used to construct them.

To state our main result (Theorem~\ref{T:unicity}), let $T$ be an anisotropic maximal torus of $G$ and let $\buil^T$ denote the fixed point set of $T$ acting on the Bruhat-Tits building $\buil(G)$ of $G$.  This set contains in particular a point $x$ which is the image of the building of $T$ in the building of $G$.  Let $c_T\geq 0$ be the \emph{simplicial radius} of $\buil^T$, relative to $x$, as defined in Section~\ref{4}.

\begin{theorem*}
Let $T$ be a tame elliptic maximal torus of $G$ and suppose $\pi$ is a supercuspidal representation of $G$ built from a datum containing $T$.  If the character $\phi_0$ of $T$ appearing in the datum has depth greater than $2c_T$, then $\pi$ satisfies the unicity conjecture relative to any maximal compact open subgroup of $G$.   Moreover, in this case, $\pi$ satisfies the strong unicity property if and only if $\buil^T$ consists precisely of the closure of a single facet in $\buil(G)$.
\end{theorem*}

We note that the theorem holds without any hypothesis on the depth of $\phi_0$ when $T$ is \emph{unramified}; see Corollary~\ref{C:unramified}.

An essential ingredient of the proof is the analysis of the fixed points of $\buil(G)$ under the action of both the torus $T$, and of the inducing subgroup $\KT$, using particularly Lemma~\ref{L:stabilizers}.  As we discuss in Section~\ref{S:nonunicity}, while the hypothesis for strong unicity (for example, that $\buil^T = \overline{\{x\}}$) holds in many cases (notably, for $G=\SL_n$, for unramified tori, and for purely ramified Coxeter tori), it fails for many classes of anisotropic tori in general. In these cases, the number of inequivalent types of the form $(G_y,\lambda)$ can grow arbitrarily large as the rank of $G$ increases; see Example~\ref{Ex:2}.

This paper is organized as follows.  We establish our notation in Section~\ref{1}, including particularly of the building $\buil(G)$ and of Moy-Prasad filtration subgroups, and recall the definition of the essentially tame toral supercuspidal representations we study in Section~\ref{S:construction}.   In Section~\ref{strategy} we prove a proposition about the inequivalence of toral supercuspidal representations under certain twists, based on \cite[Cor 6.10]{HakimMurnaghan2008} (recalled here as Lemma~\ref{HMCor}) and some results of F.~Murnaghan in \cite{Murnaghan2011}.  In Section~\ref{4} we establish a key result relating Moy-Prasad filtration subgroups to stabilizers of subsets of an apartment, generalizing a proposition in \cite{Nevins2014}, and we define the notion of the simplicial radius of a bounded subset of $\buil(G)$.  In Section~\ref{S:Mackey} we recall the Mackey decomposition of $\pi|_{G_y}$ for a vertex $y$ and prove two general results about components that can contain types.  
Our main results on unicity are proven in Section~\ref{S:unicity}.  We discuss the success and failure of strong unicity, and provide examples where it fails, in Section~\ref{S:nonunicity}.

Several interesting problems remain open.  The hypothesis on the depth of the character of $T$ given in the statement of Theorem~\ref{T:unicity} arises as a result of our method of proof.  There is no such restriction for the unicity theorems on $G=\GL_n$ or on $G=\SL_n$, where a different argument (particular to type A) was employed by V.~Pa\u{s}k\={u}nas to address the small-depth case.  Finally, not all essentially tame supercuspidal representations arise from toral data; the consideration of general twisted Levi sequences is necessary to completely resolve the unicity conjecture in these cases.

\subsection*{Acknowledgements} The second author warmly thanks Anne-Marie Aubert, Manish Mishra, Alan Roche and Steven Spallone for the invitation to the excellent conference \emph{Representation theory of $p$-adic groups} at IISER Pune, India.   The stimulating environment of the workshop contributed significantly to this article; in particular, fellow participant Jeff Adler provided invaluable insight into tori and buildings, and he offered up the torus of Example~\ref{Ex:1}.

\section{Notation}\label{1}

Let $\ratk$ be a nonarchimedean local field of residual characteristic $p$, with integer ring $\R$ and maximal ideal $\PP$.  We normalize the valuation 
$\nu$ on $\ratk$ so that $\nu(\ratk^\times) =\Z$; if $\extk$ is an extension field of $\ratk$ then we also denote by $\nu$ the unique extension of this valuation to $\extk$.  Fix an additive character $\Lambda$ of $\ratk$ that is nontrivial on $\R$ but trivial on $\PP$.  If $H\subset G$ are groups and $g\in G$, let ${}^gH = \{ghg^{-1}\mid h\in H\}$ and for any representation $\tau$ of $H$ let ${}^g\tau$ denote the corresponding representation of ${}^gH$.  If $\sigma$ is a representation of $G$ we write $\Res_H\sigma$ for its restriction to $H$.

Let $\GG$ be a semisimple simply connected linear algebraic group defined over $\ratk$ and let $G = \GG(\ratk)$.   Let $\buil(G)=\buil(\GG,\ratk)$ denote the reduced building of $\GG$ over $\ratk$; since $\GG$ is semisimple it coincides with the enlarged building.  Since $\GG$ is simply connected, the stabilizer of a point $x$ in the building, $G_x$, coincides with the parahoric subgroup $G_{x,0}$ associated to the facet containing $x$.  Both hypotheses together imply that the maximal compact open subgroups of $G$ are exactly the maximal parahoric subgroups of $G$, that is, $G_y$ for each vertex $y$ of the building.

Let $\Stor$ be a maximal $\ratk$-split torus in $\GG$ defined over $\ratk$.  Fix a maximal unramified extension $\unk$ of $\ratk$, and let $\Stor'$ be a maximal $\unk$-split torus of $\GG$ defined over $\ratk$ and containing $\Stor$.  Let $\Ztor$ be the centralizer of $\Stor'$ in $\GG$, which is a maximal torus of $\GG$ defined over $\ratk$.
  Denote by $\Phi = \Phi(\GG,\Stor)$ the root system of $\GG$ relative to $\Stor$ and by $\apart = \apart(\Stor, \ratk)$ the apartment of $\buil(\GG,\ratk)$ corresponding to $\Stor$.   Let $\Psi_\Phi$ be the set of affine roots, corresponding to our choice of valuation $\nu$; these are functions on $\apart$. 
Denote the root subgroup of $\GG$ corresponding to $\alpha \in \Phi$ by $\GG_\alpha$.   For $\alpha \in \Phi$ we set $G_\alpha=\GG_{\alpha}(\ratk)$; this group admits a filtration by compact open subgroups $G_\psi$, as $\psi$ runs over the affine roots with gradient $\dot\psi=\alpha$.  For more details, see \cite[\S 4]{BruhatTits1984} or the careful exposition in \cite[\S 2]{Fintzen2015}.
 
 Let $S=\Stor(\ratk)$, $S'=\Stor'(\ratk)$ and $Z=\Ztor(\ratk)$.  Let $\Zb$ be the maximal bounded subgroup of $Z$.   Recall that any torus $T=\TT(\ratk)$ admits a natural filtration by subgroups $T_r$ for $r\geq 0$, and its Lie algebra $\LieT$ is filtered by lattices $\LieT_r$ for $r\in \mathbb{R}$.  In particular, $Z$ admits a natural filtration by subgroups $Z_{r}$ for $r\geq 0$, with $Z_0\subseteq \Zb$. 
 
 For any $x\in \apart$ and $r\geq 0$, A.~Moy and G.~Prasad \cite{MoyPrasad1994} defined $G_{\alpha,x,r} = \langle G_\psi \mid \psi(x)\geq r, \dot\psi=\alpha \rangle$ and thus filtration subgroups
 $$
 G_{x,r}=\langle Z_{r}, G_{\alpha,x,r} \mid \alpha \in \Phi\rangle 
  $$
  of $G_x$.
We set $G_{x,r+} = \cup_{s>r}G_{x,s}$.  They similarly defined lattices $\LieG_{x,r}$ in $\LieG=\Lie(G)(\ratk)$ and $\LieG^*_{x,r}$ in  $\LieG^*$, indexed by $r\in \mathbb{R}$.  Conjugation by $G$ allows us to extend these definitions to any $x\in \buil(G)$.
  
We say that a group $\GG'$ (or its set of $\ratk$-points $G' = \GG'(\ratk)$) is a \emph{twisted Levi subgroup} of $\GG$ if there is an extension $\extk$ of $\ratk$ such that  $\GG'$ is an $\extk$-Levi subgroup of $\GG$ defined over $\ratk$.   We say $\GG'$ is tamely ramified if $\GG'$ (and thus $\GG$) splits over a tamely ramified extension.

Suppose now that $\GG'$ is a (tamely ramified) twisted Levi subgroup of $\GG$ and let $\TT$ be a maximal torus of $\GG'$; let $\extk$ be a tamely ramified splitting field of $\TT$ over $\ratk$.
Let $\LieG'$ denote the $\ratk$-points of $\Lie(\GG')$, and denote by ${\mz'}^*$ the $\ratk$-points of the dual of the center of $\Lie(\GG')$.

Let $r>0$.  Following \cite[Definition 3.7]{HakimMurnaghan2008}, an element $X^* \in {\mz'}^*_{-r}$ is called \emph{$G$-generic of depth $-r$} if it satisfies the conditions \textbf{GE1} and \textbf{GE2} of \cite[\S8]{Yu2001}. By \cite[Lemma 8.1]{Yu2001}, if $p$ is not a torsion prime for the root datum (in the sense of \cite{Springer1977}) of $(\GG,\TT)$, then these conditions reduce to the requirement that $\nu(\langle X^*, H_a\rangle) = -r$, for each $a \in \Phi(\GG,\TT)\setminus\Phi(\GG',\TT)$, where $H_a \in \Lie(\GG)(\extk)$ is the coroot associated to $a$.   

 Fix also a Moy-Prasad isomorphism \cite[\S 2.6]{HakimMurnaghan2008}
$$
e \colon G'_{x,r}/G'_{x,r+} \to \LieG'_{x,r}/\LieG'_{x,r+}.
$$
A character $\phi$ of $G'$ of depth $r$ is said to be \emph{realized by an element $X^* \in {\mz'}^{*}_{-r}$ on $G'_{x,r}$} if 
for every $Y \in \LieG'_{x,r}$, 
$$
\phi(e(Y+\LieG'_{x,r+})) = \Lambda(\langle X^*, Y\rangle).
$$
By \cite[Definition 3.9]{HakimMurnaghan2008}, $\phi$ is \emph{$G$-generic (relative to $x$) of depth $r>0$} if $\phi$ is trivial on $G'_{x,r+}$, nontrivial on $G'_{x,r}$, and is realized on $G'_{x,r}$ by an element $X^* \in {\mz'}^{*}_{-r}$ that is $G$-generic of depth $-r$.  
A particular consequence of the $G$-genericity of a character $\phi$ of $G'$ is that for any $g\in G$, we have that ${}^g\phi$ and $\phi$ coincide on $G'_{x,r}$ if and only if $g\in G'$ \cite[Lemma 8.3]{Yu2001}.

 \section{The construction of toral supercuspidal representations}\label{S:construction}

Fix $G$ semisimple and simply connected, and retain all the notation above.  Following \cite[\S3]{Yu2001} and \cite[\S3.1]{HakimMurnaghan2008}, we define a \emph{positive-depth generic toral supercuspidal datum of length $d$} of $G$ (hereafter: \emph{toral supercuspidal datum}) to be a pair $(\vec{G}, \vec{\phi})$, where
 \begin{itemize}
\item $\vec{G} = (G^0\subsetneq G^1 \subsetneq \cdots \subsetneq G^d=G)$, for some $d \geq 1$, is a sequence of tamely ramified twisted Levi subgroups, such that in particular $G^0=T$ an anisotropic maximal torus in $G$;  
\item $\vec{\phi} = (\phi_0, \phi_1, \ldots, \phi_{d-1})$ where $\phi_i$ is a $G^{i+1}$-generic character of $G^i$ of depth $r_i$, and these real numbers satisfy $0<r_0<r_1 < \cdots < r_{d-1}$.
\end{itemize}
For convenience we write $s_i = r_i/2$ for each $i = 0, \ldots, d-1$.   
Given a tamely ramified twisted Levi subgroup $\GG'$ of $\GG$ that is defined over $\ratk$,  one may embed $\buil(\GG',\ratk)$ into $\buil(\GG,\ratk)$ with canonical image.  Since $T$ is anisotropic, the image of $\buil(T)$ in $\buil(G)$ consists of a single point $x$, and it is this point relative to which the characters $\phi_i$ are $G^{i+1}$-generic. 
Thus a toral supercuspidal datum implies the datum $(\vec{G}, \vec{\phi}, d, \vec{r}, \vec{s}, x)$, and we will take this extra data for granted where there is no possibility of confusion.

Let $\Psi = (\vec{G},\vec{\phi})$ be a toral supercuspidal datum.  Set $\KT^0=G^0_x=T$.  For each $1\leq i\leq d$ define  
$$
\KT^i:= G^0_{x}G^1_{x,s_0}\cdots G^{i}_{x,s_{i-1}} 
\subseteq TG_{x,s_0} \subseteq G_x.
$$
The groups in these products have large pairwise intersections, so we next define subgroups $\J^i$ and $\J^i_+$ which will have the property that $\KT^i\J^{i+1}=\KT^{i+1}$ and $\KT^i\cap \J^{i+1} = \KT^i\cap \J^{i+1}_+ = G^i_{x,r_i}$. 

Let $\extk$ be a tamely ramified Galois extension of $\ratk$ over which $\TT$ splits.  For each $0\leq i\leq d$, let $\Phi_i = \Phi(\GG^i,\TT)$ be the corresponding root system.   For $0 \leq i < d$ let $\J(\extk)^{i+1}$ be the group generated by $\TT(\extk)_{r_{i}}$ together with the root subgroups  
$\GG_a(\extk)_{x,r_{i}}$ for $a\in \Phi_{i}$ and $\GG_b(\extk)_{x,s_{i}}$ for $b\in \Phi_{i+1}\setminus \Phi_{i}$.  Similarly, let
$$
\J(\extk)^{i+1}_+ = \langle \TT(\extk)_{r_{i}}, \GG_a(\extk)_{x,r_{i}}, \GG_b(\extk)_{x,s_{i}+} \mid a\in \Phi_{i}, b\in \Phi_{i+1}\setminus \Phi_{i}  \rangle.
$$
 These subgroups of $\GG^{i+1}(\extk)$ are invariant under $\Gal(\extk/\ratk)$ and we set
$\J^{i+1}=\J(E)^{i+1}\cap G^{i+1}$, $\J^{i+1}_+ = \J(\extk)^{i+1}_+\cap G^{i+1}$.

Next for each $0\leq i <d$, the subdatum $((G^i,G^{i+1}), \phi_i)$ is used to construct a representation $\phi_i'$ of $\KT^{i+1}$.  The first step is to extend the restriction of $\phi_i$ to $G^i_{x,r_i}$ trivially to a character $\hat\phi_i$ of  $J^{i+1}_+$ \cite[\S4]{Yu2001}.
 Next, when $\J^{i+1}_+=\J^{i+1}$, define a character $\phi_i'$ of $\KT^{i+1}$ by 
\begin{equation}\label{E:phi'1}
\phi_i'(kj)=\phi_i(k)\hat\phi_i(j),
\end{equation}
for each $k\in \KT^i$, $j\in \J^{i+1}$.
 On the other hand, when $\J^{i+1}_+ \neq \J^{i+1}$, then instead one uses $\hat\phi_i$ 
 to construct a Heisenberg-Weil representation $\omega$ of $\KT^i\ltimes \J^{i+1}$, and then define the representation $\phi_i'$ on $k\in \KT^i$, $j\in \J^{i+1}$ by
\begin{equation}\label{E:phi'2}
\phi_i'(kj) = \phi_i(k)\omega(k,j).
\end{equation}

Now set $\KT = \KT^d$.  There is a well-defined inflation process extending each $\phi_i'$ trivially across the remaining subgroups to give a representation of $\KT$, which we denote $\kappa_i=\inf_{\KT^{i+1}}^{\KT}\phi_i'$  \cite[\S3.4]{HakimMurnaghan2008}. Note that this representation $\kappa_i$ is independent of any other characters $\phi_j$, $j\neq i$, in the datum \cite[Proposition 3.26]{HakimMurnaghan2008}.

Putting these together, we obtain a representation $\kappa(\Psi)= \kappa_{0}\otimes \cdots \otimes \kappa_{d-1}$ of $\KT$ such that $\pi (\Psi) = \cind_{\KT}^G \kappa(\Psi)$ is an irreducible supercuspidal representation of $G$ of depth $r=r_{d-1}$ \cite[Theorem 15.1]{Yu2001}.  It then follows that $(\KT,\kappa(\Psi))$ is a type for $\pi(\Psi)$.

\section{Results on equivalence}\label{strategy}

We begin by noting when two toral supercuspidal data give rise to equivalent representations, from \cite[Cor 6.10]{HakimMurnaghan2008}.     

\begin{lemma}[Hakim-Murnaghan] \label{HMCor}
Let $\Psi = (\vec{G},\vec{\phi})$ and $\dot\Psi = (\vec{\dot{G}}, \vec{\dot{\phi}})$ be two toral supercuspidal data.  Then  $\pi(\Psi) \cong \pi(\dot\Psi)$ if and only if there exists some $g\in G$ such that $G^0 = {}^g\dot{G}^0$ and $\Res_{G^0}(\phi_0 \phi_1 \cdots \phi_{d-1})=\Res_{G^0}{}^g(\dot{\phi}_0 \dot{\phi}_1 \cdots \dot{\phi}_{d-1})$.   In this case, $G^i={}^g\dot{G}^i$ for all $i\geq 0$.
\end{lemma}

Now suppose $\Psi = (\vec{G},\vec{\phi})$ is a toral supercuspidal datum of length $d$ with $T=G^0$.  Let $\xi$ be a character of $T$ of depth strictly less than $r_0$.  The $G^1$-genericity of the character $\phi_0$ of $G^0$ of depth $r_0$ depends only on $\Res_{G^0_{x,r_0}}\phi$.   Since $\xi\phi_0$ and $\phi_0$ coincide on $G^0_{x,r_0}$, we conclude that $\xi\phi_0$ is also a $G^1$-generic character of $G^0$ of depth $r_0$. Thus
$$
\Psi_\xi = (\vec{G}, (\xi\phi_0, \phi_1,\ldots, \phi_{d-1}))
$$
is another toral supercuspidal datum.

\begin{lemma}\label{L:Murn}
The character $\xi$ inflates to a character $\overline{\xi}$ of $\KT$ such that
$$
\kappa(\Psi_\xi) = \overline{\xi}\;\kappa(\Psi).
$$
\end{lemma}

\begin{proof}
Since $\xi$ is a character of $T=G^0=\KT^0$ of depth less than $r_0$, it is trivial on $G^0_{x,r_0}$.  Since $\KT^1=\KT^0\J^1$ and $\KT^0\cap \J^1 = G^0_{x,r_0}$, $\xi$ may be uniquely inflated to a character $\xi' = \inf_{\KT^0}^{\KT^1}\xi$ of $\KT^1$ that is trivial on $\J^1$. In particular, for any $k\in \KT^0$, $j\in \J^{1}$, we have $\xi'(kj)=\xi(k)$.

Since $\phi_0$ and $\xi\phi_0$ coincide on $G^0_{x,r_0}$, we may denote the extension to $\J^1_+$ of their common restriction by $\hat\phi_0$ and, if applicable, the corresponding Heisenberg-Weil representation of $\KT^0 \ltimes \J^1$  by $\omega$.  
Using now  \eqref{E:phi'1} and \eqref{E:phi'2}, as applicable, to construct $(\xi\phi_0)'$ and $\phi_0'$, we deduce that in both cases $(\xi\phi_0)' = \xi' \phi_0'$.

Set $\overline{\xi} = \inf_{\KT^1}^{\KT}\xi'$; then  $\inf_{\KT^1}^{\KT}\xi'\phi_0' = \overline{\xi}\kappa_0$.  Since for $i>0$, $\kappa_i$ depends only on the character $\phi_i$, we conclude that $\overline{\xi}$ factors out of the tensor product to yield $
\kappa(\Psi_\xi) = \overline{\xi}\kappa(\Psi),
$ as desired.
\end{proof}

\begin{proposition} \label{P:inequiv}
Let $\Psi = (\vec{G},\vec{\phi})$ be a toral supercuspidal datum and suppose $\xi$ is a character of $T=G^0$ of depth less than that of $\phi_0$.  If $\Psi_\xi = (\vec{G}, (\xi\phi_0, \phi_1,\ldots,\phi_{d-1}))$,  then $\pi(\Psi_\xi)$ is an irreducible supercuspidal representation of $G$ and it is equivalent to $\pi(\Psi)$ if and only if $\xi$ is trivial. 
\end{proposition}

\begin{proof}
Suppose that $\pi(\Psi) \cong \pi(\Psi_\xi)$.  In the setting of Lemma~\ref{HMCor}, this means there is some $g \in N_G(T)$ such that $\Res_{T}(\phi_0 \phi_1 \cdots \phi_{d-1})=\Res_{T}{}^g(\xi \phi_0 \phi_1 \cdots \phi_{d-1})$.   
Then ${}^g(\xi\phi_0){}^g\phi_1\cdots {}^g \phi_{d-1}$ is a refactorization of $\phi_0\cdots \phi_{d-1}$, in the sense of \cite[Definition 5.3]{Murnaghan2011}, and so by \cite[Proposition 5.4]{Murnaghan2011}, for each $i>0$, $\prod_{j=i}^d \phi_j$ and $\prod_{j=i}^d {}^g\phi_j$ coincide on $G^i_{x,r_{i-1}+}$.  
The genericity of the characters $\phi_j$ recursively implies, as in the proof of \cite[Proposition 5.6]{Murnaghan2011}, that $g\in G^1$, and so ${}^g\phi_i=\phi_i$ for all $i\geq 1$.
Therefore our equality reduces to $\Res_{T_{r_0}} \phi_0 = \Res_{T_{r_0}} {}^g(\xi\phi_0) =  \Res_{T_{r_0}} {}^g\phi_0$, whence $g\in T$ by genericity.  Returning to the first equality we conclude that $\xi$ is the trivial character of $T$.
\end{proof}

\section{Stabilizers}\label{4}

 Let $S,Z,\apart=\apart(S), \Phi$ be as in Section~\ref{1}.  Given a subset $\Omega \subseteq \apart$, the subgroup of $G$ that fixes $\Omega$ pointwise is generated by $\Zb$ (the maximal bounded subgroup of the maximal torus $Z$) and those $G_\psi$ satisfying $\psi(z) \geq 0$ for each $z\in \Omega$ \cite[\S 6.4]{BruhatTits1972}.  For $x,y\in \apart$ we write $[x,y]$ for the geodesic from $x$ to $y$; then $G_{[x,y]} = G_x\cap G_y$.

The following result is a generalization of \cite[Prop 3.3]{Nevins2014}, and it allows us to relate Moy-Prasad filtration subgroups at $x$ to stabilizers of  its neighbourhoods in any apartment $\apart$ containing $x$.  Note that $\apart$ is the affine space under $X_*(S)\otimes_{\mathbb{Z}}\mathbb{R}$.  Given two points $x,z\in \apart$, we identify $z-x$ with a vector in $X_*(S)\otimes_{\mathbb{Z}}\mathbb{R}$, and then for each $\alpha \in \Phi$, $\alpha(z-x)$ is a well-defined real number.

\begin{lemma}\label{L:stabilizers}
Let $\apart=\apart(\Stor,\ratk)$ be an apartment in $\buil(\GG,\ratk)$ containing $x$ and let $\Phi=\Phi(\GG,\Stor)$ be the corresponding root system.  Let $s \geq 0$ and define 
$$
\Aset(x,s)=\{ z \in \apart \mid \forall \alpha \in \Phi, \alpha(z-x) \leq s\}.
$$
Then $G_{x,s} \subseteq G_{\Aset(x,s)} = G_{\overline{\Aset(x,s)}}$, where 
$\overline{\Aset(x,s)}$ denotes the simplicial closure of $\Aset(x,s)$ in $\apart$.
\end{lemma}

\begin{proof}
That $G_{\Aset(x,s)} = G_{\overline{\Aset(x,s)}}$ follows from \cite[2.4.13]{BruhatTits1972}, since $G$ is semisimple and simply connected.   
Let $z\in \Aset(x,s)$.  For each affine root $\psi$ such that $\psi(x)\geq s$, let $\alpha$ denote its gradient.  Then $\psi(x)-\psi(z) = \alpha(x-z) \leq s$ so $\psi(z) \geq \psi(x)-s \geq 0$.  Thus for all affine roots $\psi$, if $G_\psi \subset G_{x,s}$ then $G_\psi \subset G_{\Aset(x,s)}$.  Since also $Z_s \subseteq \Zb$, we conclude $G_{x,s}\subseteq  G_{\Aset(x,s)}$.  
\end{proof}

More generally, we may define
$$
\Bset(x,s) = \bigcup_{\apart \ni x} \Omega_\apart(x,s)
$$
which is the union over all apartments of $\buil(G)$ containing $x$.  By local compactness, this is reduced to a finite union.  Let $\overline{\Bset(x,s)}$ denote its simplicial closure.

\begin{definition}\label{Def:simplicialradius}
Let $\Bset \subset \buil(G)$ be a bounded convex set and suppose $x\in \Bset$.  For each apartment $\apart = \apart(\Stor,\ratk)$ containing $x$ define
$$
r(\Bset, \apart, x) = \sup\{\alpha(z-x) \mid z \in \Bset\cap\apart, \alpha \in \Phi(\GG,\Stor)\}.
$$
Then the \emph{simplicial radius} of $\Bset$ with respect to $x$ is defined to be
$$
c(\Bset,x)  = \sup\{r(\Bset, \apart, x) \mid x \in \apart\}.
$$
\end{definition}

As one motivating example, note that the simplicial radius of $\Bset(x,s)$ with respect to $x$ is $c(\Bset(x,s),x)=s$. For another, letting $\overline{\{x\}}$ denote the simplicial closure of $\{x\}$ in $\buil(G)$, we have $c(\overline{\{x\}},x)<1$ since 
$\overline{\{x\}}$ is constrained between adjacent affine root hyperplanes in any apartment containing $x$.  Note that in each of these examples we have the equality $c(\Bset,x)=r(\Bset,\apart,x)$ for each apartment $\apart$ containing $x$.

One can be slightly more precise about $c(\overline{\{x\}},x)$ when $x$ arises as the point identified with $\buil(T)$ from a datum $\Psi$.  In this case, $x$ is an \emph{optimal point} of $\buil(G)$, in the sense of \cite[\S6.1]{MoyPrasad1994}, whence the family of such values $c_x=c(\overline{\{x\}}, x)$ could be computed for any $G$.  For example, if $G=\SL_n(\ratk)$, then the optimal points are among the barycentres of the facets, whence if $x$ lies in a $k$-dimensional facet $\F$ then $c_x= 1-\frac1k$.

\section{Mackey decomposition and strategies for identifying types}\label{S:Mackey}

Let $y$ be a vertex of $\buil(G)$ and $G_y$ the corresponding maximal compact open subgroup of $G$.  Let $\Psi = (\vec{G},\vec{\phi})$ be a toral supercuspidal datum with $T=G^0$ and let $\pi = \pi(\Psi)$.  We are interested in the irreducible representations of $G_y$ occurring in $\Res_{G_y}\pi$.

Mackey theory gives a decomposition
$$
\Res_{G_y}\pi  \cong \bigoplus_{g \in G_y\backslash G/\KT}\Ind_{G_y\cap {}^g\KT}^{G_y} {}^g\kappa,
$$
where each \emph{Mackey component} $\tau(g):=\Ind_{G_y\cap {}^g\KT}^{G_y} {}^g\kappa$ is a finite-dimensional representation of $G_y$.  Note that since $G_y=N_G(G_y)$, the Mackey components are parametrized by a subset of the $G$-orbit of the vertex $y$ in $\buil(G)$. 
We emphasize that these Mackey components are not, in general, irreducible; a first strategy for identifying those that contain types is the following.

\begin{proposition}\label{P:type}
If $\KT \subseteq G_{g^{-1}y}$ then $(G_y,\tau(g))$ is a type for $\pi$.
\end{proposition}

\begin{proof}
Consider instead the twisted Mackey component
$$
 {}^{g^{-1}}\tau(g) = {}^{g^{-1}}(\Ind_{G_y\cap {}^g\KT}^{G_y} {}^g\kappa) \cong \Ind_{G_{g^{-1}y}\cap \KT}^{G_{g^{-1}y}} \kappa
$$
which is a representation of $G_{g^{-1}y}$.  If $\KT\subseteq G_{g^{-1}y}$, then ${}^{g^{-1}}\tau(g) \cong \Ind_{\KT}^{G_{g^{-1}y}}\kappa$.  By the transitivity of compact induction, $\cind_{G_{g^{-1}y}}^G\tau' = \pi$, whence by Frobenius reciprocity $(G_{g^{-1}y},{}^{g^{-1}}\tau(g))$ is a type for $\pi$, and the result follows.
\end{proof}

On the other hand, the key strategy to discern Mackey components that cannot contain a type of $\pi$ is the following.

\begin{theorem} \label{T:inequiv}
Given $y, \Psi, g$ as above, suppose that there exists a nontrivial character $\xi$ of $T$ of depth less than $r_0$ such that its inflation $\overline{\xi}$ to $\KT$ is trivial on $\KT \cap G_{g^{-1}y}$.   
Then no irreducible subrepresentation of the Mackey component
$$
\Ind_{G_y \cap {}^g\KT}^{G_y} {}^g\kappa(\Psi)
$$
is a type.
\end{theorem}

\begin{proof}
The twisted Mackey component ${}^{g^{-1}}\tau(g)$ depends only $\Res_{G_{g^{-1}y}\cap \KT}\kappa(\Psi)$.  By Lemma~\ref{L:Murn}, $\kappa(\Psi_\xi)=\overline{\xi}\kappa(\Psi)$; by hypothesis 
$\Res_{G_{g^{-1}y}\cap\KT}\overline{\xi}$ is trivial.  
Thus 
$\Ind_{G_y \cap {}^g\KT}^{G_y} {}^g\kappa(\Psi)$ is a common component of both $\pi(\Psi)$ and $\pi(\Psi_\xi)$.  By Proposition~\ref{P:inequiv}, $\pi(\Psi)\not\cong\pi(\Psi_\xi)$.  Since these are inequivalent irreducible supercuspidal representations, they have distinct inertial support, and thus no irreducible representation of $G_y$ occurring in $\Ind_{G_y \cap {}^g\KT}^{G_y} {}^g\kappa(\Psi)$ can be a type.
\end{proof}

\section{Unicity results}\label{S:unicity}

We first identify the obvious types occurring in $\Res_{G_y}\pi$.

\begin{lemma} \label{L:obvious}
Suppose $y$ is a vertex of the facet containing $x$.  Then $(G_y,\tau(1))$ is a type. 
\end{lemma}

\begin{proof}
Let $\F$ be the facet containing $x$.  We have $\KT\subseteq G_x$.  Since $y \in \overline{\F}$ by hypothesis, $G_x \subseteq G_y$.  This implies that $\KT\subseteq G_y$, whence the first statement by Proposition~\ref{P:type}.  
\end{proof}

Now, continuing with the notation of the previous section, we identify $g\in G$ for which the associated Mackey components do not contain types.

\begin{proposition}\label{P:line}
Let $g \in G$. Suppose that there exists $z\in [x,g^{-1}y]\cap \overline{\Bset(x,s_0)}$ such that $T \not\subseteq G_z$.  Then no irreducible subrepresentation of $\tau(g)$
is a type.
\end{proposition}

\begin{proof}
Let $\apart$ be an apartment of $\buil(G)$ containing $x$ and $g^{-1}y$; then $\apart$ contains the geodesic $[x,g^{-1}y]$.  Let $z\in [x,g^{-1}y]\cap \overline{\Aset(x,s_0)}$ as in the proposition.  Since $z$ is on the line $[x,g^{-1}y]$, whose pointwise stabilizer is $G_x\cap G_{g^{-1}y}$, we have that $G_x\cap G_{g^{-1}y} \subseteq G_z$.  At the same time, Lemma~\ref{L:stabilizers} implies that $G_{x,s_0}\subseteq G_{\overline{\Aset(x,s_0)}}\subseteq G_z$.  Thus
$$
\KT \cap G_{g^{-1}y} \subseteq TG_{x,s_0} \cap G_x \cap G_{g^{-1}y} \subseteq TG_{x,s_0}\cap G_z = (T\cap G_z)G_{x,s_0}.
$$
 Noting that $T_{s_0} \subseteq T\cap G_z$, we deduce that any character of $T$ that is trivial on $T\cap G_z$ has depth strictly less than $s_0$.  
 
By hypothesis $T\cap G_z \subsetneq T$; so let $\xi$ be a nontrivial character of $T$ that is trivial on $T\cap G_z$.  Since its depth is less than $s_0$, its inflation $\overline{\xi}$ is trivial on $(T\cap G_z)G_{x,s_0}$.  We deduce that $\xi$ satisfies the hypotheses of Theorem~\ref{T:inequiv}, whence the result.
\end{proof}

Let 
$$
\buil^T = \{z\in \buil(G) \mid T\subseteq G_z\}
$$
be the set of fixed points of $T$ acting on $\buil(G)$.  This is a convex subset of $\buil(G)$ containing $x$, and it is compact since $T$ is an anisotropic  maximal  torus.   Let $c_T$ denote the simplicial radius of $\buil^T$ with respect to $x$, as per Definition~\ref{Def:simplicialradius}.

Now let $\apart$ be an apartment containing $x$.  The hypothesis of Proposition~\ref{P:line} will be satisfied for all $g^{-1}y\in \apart$ if $\apart^T=\apart \cap \buil^T$ is contained in the interior of the simplicial closure $\overline{\Aset(x,s_0)}$ of $\Aset(x,s_0)$.   In particular this holds if $s_0> c_T \geq r(\buil^T,\apart,x)$.  

Putting these geometric ideas together yields our main theorem.

\begin{theorem}\label{T:unicity}
Let $T$ be a tamely ramified  anisotropic maximal torus of $G$ and let $s_0>c_T$.  Let $\Psi = (\vec{G},\vec{\phi})$ be any toral supercuspidal datum such that $G^0=T$ and such that the depth of $\phi_0$ is at least $2s_0$.  Then $\pi(\Psi)$ satisfies the conjecture of unicity of types relative to any maximal compact open subgroup of $G$.  If moreover $\buil^T$ is the closure of a single facet of $\buil(G)$ then $\pi(\Psi)$ has the property of strong unicity.
\end{theorem}

\begin{proof}
Let $g\in G_y\backslash G / \KT$.  Suppose first that $g^{-1}y \in \buil^T$, so that $T\subseteq G_{g^{-1}y}$.  Since $s_0>c_T$, we have $g^{-1}y \in \Bset(x,s_0)$ so $G_{x,s_0}\subseteq G_{g^{-1}y}$ as well.  Thus $\KT \subseteq TG_{x,s_0} \subseteq G_{g^{-1}y}$ and by the same argument as 
Lemma~\ref{L:obvious}, we conclude that $(G_{y},\tau(g))$ is a type for $\pi$.

Now suppose that $g^{-1}y \notin \buil^T$, so that $T \not\subset G_{g^{-1}y}$.  
Since $s_0>c_T$,  $\Bset(x,s_0)$ contains an open neighbourhood of $\buil^T$, so the line $[x,g^{-1}y]$ must meet $\Bset(x,s_0) \setminus \buil^T$ in at least one point $z$.  Using now Proposition~\ref{P:line}, we conclude that the corresponding Mackey component contains no types for $\pi(\Psi)$.

Finally, we note that if $\buil^T = \overline{\F}$ for a facet $\F\subset\buil(G)$, then since $G$ is semisimple simply connected, each orbit of a vertex $y$ in $\buil(G)$ meets $\buil^T$ at most once.  By the above arguments, this implies strong unicity.
\end{proof}

By \cite[3.6.1]{Tits1979}, if $T$ is an anisotropic maximal  torus which splits over an unramified extension, then $\buil^T$ consists of a single vertex, namely $\{x\}$, whence $c_T=0$.  Since all our toral supercuspidal representations have positive depth, we have the following immediate corollary.

\begin{cor}\label{C:unramified}
Let $T$ be an \emph{unramified}  anisotropic maximal torus of $G$.  Let $\Psi = (\vec{G},\vec{\phi})$ be any toral supercuspidal datum such that $G^0=T$.  Then $\pi(\Psi)$ satisfies  strong unicity of types relative to any maximal compact open subgroup of $G$.  
\end{cor}

\section{The inequivalence of unicity and strong unicity}\label{S:nonunicity}

Recall that strong unicity of types is the statement that $\Res_{K}\pi$ should contain at most one type, for each choice of maximal compact open subgroup $K$.   We have the following converse to the strong unicity statement in Theorem~\ref{T:unicity}.  Note that this result is without a condition on depth.

\begin{lemma}\label{L:strongunicityfails}
Let $T$ be a tamely ramified anisotropic maximal  torus.  If $\buil^T$ contains two distinct but $G$-conjugate vertices $y$ and $y'$ of alcoves whose closure contain $x$, then for any toral supercuspidal datum $\Psi = (\vec{G},\vec{\phi})$ such that $G^0=T$, the strong unicity property fails to hold for $\pi(\Psi)$, that is, there exist at least two nonisomorphic types in $\Res_{G_y}\pi(\Psi)$.
\end{lemma}

\begin{proof}
Let $x \in \buil(T) \subseteq \buil(G)$ and set $\pi=\pi(\Psi)$.  Let  $y\neq y'$ be $G$-conjugate vertices of two distinct alcoves $\C, \C'$ of $\buil(G)$ such that $x \in \overline{\C}\cap \overline{\C'}$.  Then for any $s_0>0$, we have that $y,y' \in \overline{\Bset(x,s_0)}$.  Let $g \in G$ be such that $y'=g^{-1}y$.  If both $y$ and $g^{-1}y$ lie in $\buil^T$, then by the proof of Theorem~\ref{T:unicity}, we conclude that both $(G_y,\tau(1))$ and $(G_{y},\tau(g))$ are types for $\pi$ occurring in $\Res_{G_y}\pi$.  Note that although these types are induced from the $G$-conjugate types $(\KT,\kappa)$ and $({}^g\KT,{}^g\kappa)$, they are not themselves $G$-conjugate since $g\notin N_G(G_y)=G_y$.
Neither are these types isomorphic, since $\pi \cong \cind_{G_y}^G\tau(1)$ implies
$$
\mathbb{C}\cong\Hom_{G}(\pi,\pi) \cong \Hom_{G_y}(\tau(1),\Res_{G_y}\pi)
$$
by Frobenius reciprocity.
\end{proof}

We next prove the existence of pairs $(G,T)$ satisfying the hypotheses of Lemma~\ref{L:strongunicityfails} with an example such that $G$ has rank $2$.  We thank Jeff Adler for providing this instructive example of a torus that stabilizes more than the closure of a single facet.

\begin{example}\label{Ex:1}
Consider $G=\Sp_4(\ratk)$, given in matrix form relative to $J=\left[\begin{smallmatrix}0&I\\-I&0\end{smallmatrix}\right]\in M_4(\ratk)$ 
as the set
$$
G = \{ g \in \GL_4(\ratk)\mid {}^tgJg=J\}
$$
where ${}^tg$ denotes the transpose.  Note that $G$ contains a generalized Levi subgroup $G'$ isomorphic to $\SL_2(\ratk)\times \SL_2(\ratk)$, defined by the long roots.

Let $\p$ be a uniformizer of $\ratk$.  For $i=1,2$, we choose an anisotropic torus $T_i$ of $\mathrm{SL}_2(\ratk)$ that isomorphic to the norm-one elements of some quadratic ramified extension $\extk_i=\ratk[\sqrt{\gamma_i\p}]$ with $\nu(\gamma_i)=0$.  Let $T=T_1\times T_2$ be the corresponding anisotropic torus of $G'$; explicitly we embed $T$ into $G$ as the subgroup
$$
T=\left\{ 
\bmat{a &0 &\phantom{b}b\phantom{b}&\phantom{b}0\phantom{b}\\ 
0&c&0&d\\ 
b\gamma_1\p &0&a&0\\ 0&d\gamma_2\p & 0& c}
\;\; \middle| \;\;\parbox{1.2in}{$a,b,c,d \in \R, \quad\quad$ 
$a^2-b^2\gamma_1\p = 1$, $c^2-d^2\gamma_2\p=1$}\right\}.
$$
For each $i$, the point in the building $\buil(\SL_2,\ratk)$ corresponding to $T_i$ is the midpoint of the fundamental alcove.  Thus the point $x$ representing $\buil(T)\subset \buil(G)$ is the midpoint of the diagonal facet $\F$ in the closure of the fundamental alcove $\C$ in the standard apartment $\apart$ of $\buil(G)$, as illustrated in Figure~\ref{Fig:1}.  Evidently $\overline{\F} \subseteq \apart^T$.

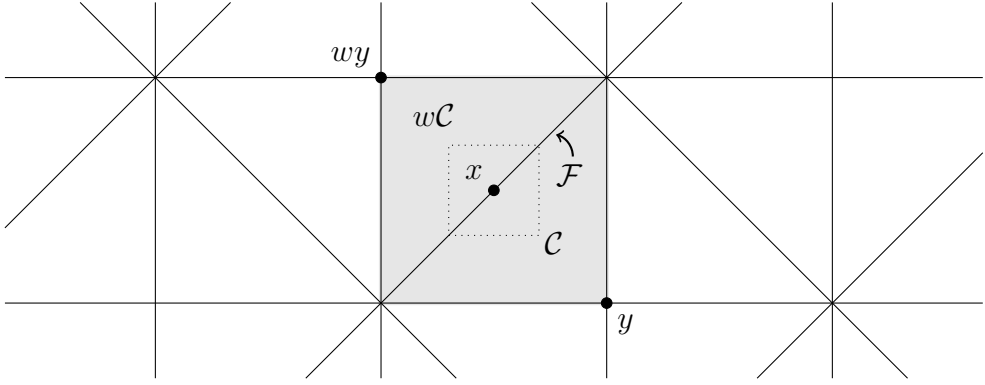
\begin{figure}[ht]
\begin{tikzpicture}
\draw[fill=gray!20,ultra thick, gray!20] (1,1) rectangle (4,4);

\draw (-4,1)--(9,1);
\draw (-4,4)--(9,4);
\draw (1,0)--(1,5);
\draw (4,0)--(4,5);
\draw (-2,0)--(-2,5);
\draw (7,0)--(7,5);
\draw (0,0)--(5,5);
\draw (-3,5)--(2,0);
\draw (3,5)--(8,0);
\draw (6,0)--(9,3);
\draw (-4,2)--(-1,5);

\draw[fill] (2.5,2.5) circle [radius=2pt];
\node[above left]  at (2.5,2.5) {$x$};

\draw[fill] (4,1) circle [radius=2pt];
\node[below right]  at (4,1) {$y$};

\draw[fill] (1,4) circle [radius=2pt];
\node[above left]  at (1,4) {$wy$};

\node at (3.3,1.8) {$\mathcal{C}$};
\node at (1.7,3.5) {$w\mathcal{C}$};

\node at (3.5,2.7) {$\F$};
\draw[thick,->] (3.55,2.95) arc (0:75:0.3cm);

\draw[dotted] (1.9,1.9) rectangle (3.1,3.1);
\end{tikzpicture}
\caption{A portion of the standard apartment $\apart$ of $\Sp_4(\ratk)$, identifying an alcove $\mathcal{C}$, the point $x$ that is the image of $\buil(T)$ in $\buil(G)$, the facet $\F$ containing $x$ and the vertex $y$ of $\C$ not in $\overline{\F}$.  The images of $y$ and of $\C$ under a reflection $w$ in the Weyl group are also indicated.  The subset $\apart^T$ of $\apart$ fixed by $T$ is the closed region shaded in gray.  For sake of example, the set $\Aset(x,s_0)$, with $s_0=\frac{1}{10}$, is indicated with a dotted line; its closure is $\apart^T$.}
\label{Fig:1}
\end{figure}

Let us now prove that $\apart^T \neq \overline{\F}$.  Let $y$ denote the (non-special) vertex of $\overline{\C}$ opposite $\F$, and $w\in G$ a reflection in the wall of $\apart$ containing $\F$ (viewed as a representative of the corresponding element of the affine Weyl group).  Adopting the convention that a matrix ring stands for its intersection with $\Sp_4(\ratk)$, we compute directly that
\begin{equation}\label{E:GCGy}
G_{\C} = \bmat{\R & \R & \R & \R \\ \PP  & \R & \R&\R \\ \PP &\PP & \R&\PP\\ \PP&\PP&\R&\R}
\quad \text{and} \quad
G_{w\C} = 
\bmat{\R & \PP & \R & \R \\ 
\R  & \R & \R&\R \\ 
\PP &\PP & \R&\R\\ 
\PP&\PP&\PP&\R},
\end{equation}
each of which contains $T$ as a subgroup.  It therefore follows that $\{y,wy\} \subset \overline{\C}\cup \overline{w\C}\subseteq \apart^T$, even though neither $y$ nor $wy$ lie in $\overline{\F}$.
Applying now Lemma~\ref{L:strongunicityfails}, we deduce that strong unicity fails for any supercuspidal representation $\pi(\Psi)$ constructed from $G^0=T$.
\end{example}

\begin{remark}\label{R:ex}
In the setting of the preceding example, we can conclude slightly more.  First note that if $z\in \apart$ and $T\subseteq G_z$, then for each positive long root $\alpha$, we must have $0\leq \alpha(z)\leq 1$.  It follows that $\apart^T = \overline{\C}\cup\overline{w\C}$, as indicated in Figure~\ref{Fig:1}.  We claim that in fact $\buil^T=\apart^T$.

Since $G_x\cap G_y=G_\C$, the orbit of $y$ is parametrized by the set $G_x/G_\C \simeq \GL_2(\R)/\mathcal{I}$, where $\mathcal{I}$ denotes its Iwahori subgroup.  A set of representatives is
$$
\left\{\overline{a} = \bmat{1&a&&\\0&1&& \\ && 1&0\\&&-a&1}, w= \bmat{0&1&&\\-1&0&&\\&&0&1\\&&-1&0} \;\;\middle|  \;\;a \in \R/\PP \right\}.
$$
One can verify directly that for each $a \in \R^\times$, ${}^{\overline{a}} T\not\subseteq G_y$.  Thus since the fixed point set $\buil^T$ is closed and convex but contains no vertices of chambers adjacent to $x$ outside of $\apart$, we conclude that $\buil^T=\overline{\mathcal{C}}\cup \overline{w\mathcal{C}}=\apart^T$, which is contained entirely in the standard apartment.   Therefore we can compute the simplicial radius of $\buil^T$ with respect to $x$, yielding $c_T=\frac12$.  

We had see that $\tau(1)$ and $\tau(w)$ gave types.  For $a\in \R^\times$, since the choice $z=\overline{a}^{-1}y$ satisfies the hypotheses of  Proposition~\ref{P:line} for any $s_0>0$, we deduce that the corresponding Mackey components $\tau(\overline{a})$ do not contain types.  For the remaining Mackey components, $g^{-1}y$ is not in the closure of an alcove adjacent to $x$, and Theorem~\ref{T:unicity} implies that if $s_0>c_T=\frac12$, any corresponding toral supercuspidal representation contains exactly the two types on $G_y$ identified above.
\end{remark}

Example~\ref{Ex:1} generalizes immediately, to give families of supercuspidal representations for which the number of distinct types supported on a given non-special maximal compact subgroup grows exponentially with the rank of $G$.

\begin{example} \label{Ex:2}  Let $n\geq 2$ and consider the subgroup $G' \cong \SL_2(\ratk)^n$ of $G=\Sp_{2n}(\ratk)$ generated by the root subgroups corresponding to long roots.  To allow us to be explicit, let the roots of $\Sp_{2n}$ with respect to the diagonal torus $\Stor$ be the set $\Phi=\{\ep_i\pm \ep_j, \pm2\ep_i \mid 1\leq i \neq j \leq n\}$, with simple system $\Delta=\{\ep_i-\ep_{i+1}, 2\ep_n\mid 1\leq i < n\}$.   With respect to the basis $\{e_1,\ldots, e_n\}$ of $X_*(\Stor)\otimes_{\mathbb{Z}}\mathbb{R}$ (whose affine space is $\apart$), dual to $\{\ep_1, \ldots, \ep_n\}$,  the vertices of the fundamental alcove $\C$ are $v_i = \sum_{j=1}^i \frac12e_j$, for $0\leq i \leq n$.  Let $W=W_{\GG}$ denote the Weyl group of $\GG$ relative to $\Stor$.

In $\SL_2(\ratk)$, there are two conjugacy classes of unramified anisotropic tori, attached to the distinct conjugacy classes of vertices in $\buil(\SL_2,\ratk)$.  There are between 2 and 4 conjugacy classes of ramified anisotropic tori, attached to the midpoint of facets; see \cite{Nevins2013}, for example.  With respect to the coordinates above, the roots of each $\SL_2(\ratk)$ subgroup are $\pm 2\ep_i$, and thus up to conjugacy we can arrange that the vertices have $e_i$-coordinates in $\{0,\frac12\}$, whereas the midpoints have $e_i$-coordinate $\frac14$.

Let each of $T_1,\ldots, T_n$ represent an anisotropic torus of $\SL_2(\ratk)$, ordered so that: for $1\leq i\leq m$, $T_i$ is unramified and attached to $\frac12$; for $m+1\leq i \leq m+\ell$, $T_i$ is ramified and attached to $\frac14$; and for $i>m+\ell$, $T_i$ is again  unramified, but attached to $0$.  Then $T=T_1\times \cdots \times T_n$ embeds as an anisotropic maximal  torus of $G$, and $\{x\} = \buil(T)\subseteq \buil(G)$ has coordinates
$$
x = \sum_{j=1}^m \frac12e_j + \sum_{j=m+1}^{m+\ell} \frac14e_j.
$$
This is the midpoint of the $1$-dimensional facet $\F$ whose closure is the geodesic $[v_m, v_{m+\ell}]$.  Let $\Omega = \{\sum_{j=1}^m \frac12e_j +  \sum_{j=m+1}^{m+\ell} a_je_j \mid 0\leq a_j \leq \frac12\}$; then $\F \subseteq \Omega$ and one can verify directly that $\Omega = \{z \in \apart \mid T\subseteq G_z\}=\apart^T$ using the argument of Remark~\ref{R:ex}.  Note that since every maximal facet of $\Omega$ has $x$ in its closure,  $\Omega \subseteq \overline{\Aset(x,s)}$ for any $s>0$.

Note that $\overline{\F} \neq \Omega$ if and only if $\ell>1$.  In this case, let $W^\Omega = \{ w\in W\mid w\Omega \subseteq \Omega\}$ and let $W_{\Omega} = \{ w\in W\mid \forall z\in \Omega, wz=z \}$.  The elements of $W(T) := W^\Omega/W_{\Omega}$ permute the vertices of $\Omega$.  Let $\GG' = \Sp_{2m}\times \SL_2^\ell \times \Sp_{2n-2m-2\ell}\subseteq \GG$; then we can identify $W(T)$ with $W_\GG/W_{\GG'}$.  
The vertices of $\Omega\cap \C$ are $v_{m}, \ldots, v_{m+\ell}$.  The vertices $v_m$ and $v_{m+\ell}$ are fixed by each element of $W(T)$, but the orbit of $y=v_{m+t}$ under $W(T)$, for $0<t<\ell$, contains $\left(\begin{smallmatrix} \ell\\ t\end{smallmatrix}\right)$ distinct vertices of $\Omega$.  

Let $\Psi$ be any toral supercuspidal datum such that $G^0=T$, and $\pi=\pi(\Psi)$.  By the argument of the proof of Lemma~\ref{L:strongunicityfails}, for each vertex $y=v_{m+t}$ with $0<t<\ell$, $\Res_{G_y}\pi$ contains $\left(\begin{smallmatrix} \ell\\ t\end{smallmatrix}\right)>1$ inequivalent types.
\end{example}

\begin{remark}
The quadratic tori of Examples~\ref{Ex:1} and \ref{Ex:2} are the smallest of a broad class of tori to which the preceding arguments apply.  For example, using L.~Morris's classification of anisotropic maximal  principal tori of $\Sp_{2n}(\ratk)$ in \cite{Morris1991} one can explicitly construct products of tori of arbitrary rank, obtaining analogous results.  Similar constructive arguments may be made for groups of type $B_n$ and $D_n$.  
\end{remark}

The question of determining the fixed points of an anisotropic maximal  torus $T$ acting on $\buil(G)$ was partially addressed by F.~Hurst in his thesis \cite{Hurstthesis}.  
Under the hypotheses that $\GG$ is simple, connected and split over $\ratk$ (but not of type $E_7$ or $E_8$), that $T$ splits over a purely tamely ramified cyclic Galois extension, and some mild assumptions on $\ratk$, Hurst proves that $\buil^{T_{0+}} = \overline{\{x\}}$ if and only if $T$ is a Coxeter torus (in which case $x$ lies in an alcove) \cite[Satz 13.14]{Hurstthesis}.  For other $T$, he shows that $x$ lies on the wall of an alcove.  Then it follows that $\buil^{T_{0+}}$ contains the closure of the $G_x$-orbit of this alcove in $\buil(G)$, since $T_{0+}\subseteq G_{x,0+} \subseteq G_{\C}$ for all alcoves $\C$ adjacent to $x$.  (This is deduced via more explicit arguments in \cite[Satz 13.15]{Hurstthesis}.)

F.~Hurst computes a range of examples in \cite[\S13]{Hurstthesis}, including of particular interest one in type $F_4$ (labeled $6_p6_p$) where $T_{0+}$ fixes pointwise an alcove that is \emph{not} adjacent to $x$.  Hence, this example may yield a torus $T$ with a more interesting set of fixed points $\buil^T$, and consequently be an unusual example to explore.

\bibliographystyle{amsplain}
\bibliography{MNevinsPLatham}

\end{document}